\documentclass[a4paper,11pt]{article}
 
\usepackage{latexsym,amssymb,amsmath, amsthm}  

\newtheorem{thm}{Theorem}
\newtheorem{lem}{Lemma}
\newtheorem{prop}{Proposition}

\newtheorem{defn}{Definition} 
\newtheorem{rem}{Remark}

\newcommand{\sgn}{\operatorname{sgn}}

\begin{document}

\title{A Cauchy kernel for the Hermitian submonogenic system}

\author{Fabrizio Colombo$^{\text{a}}$\\
\small{e-mail: fabrizio.colombo@polimi.it}
\and Dixan Pe\~na Pe\~na$^{\text{a,}}$\footnote{Marie Curie fellow of the Istituto Nazionale di Alta Matematica (INdAM)}\\
\small{e-mail: dixanpena@gmail.com}
\and Frank Sommen$^{\text{b}}$\\
\small{e-mail: fs@cage.ugent.be}}

\date{\small{$^\text{a}$Dipartimento di Matematica, Politecnico di Milano\\Via E. Bonardi 9, 20133 Milano, Italy\\\vspace{0.2cm}
$^{\text{b}}$Clifford Research Group, Department of Mathematical Analysis\\Faculty of Engineering and Architecture, Ghent University\\Galglaan 2, 9000 Gent, Belgium}}

\maketitle

\begin{abstract}
\noindent Hermitian monogenic functions are the null solutions of two complex Dirac type operators. The system of these complex Dirac operators is overdetermined and may be reduced to constraints for the Cauchy datum together with what we called the Hermitian submonogenic system (see \cite{NDS1,NDS2}). This last system is no longer overdetermined and it has properties that are similar to those of the standard Dirac operator in Euclidean space, such as a Cauchy-Kowalevski extension theorem and Vekua type solutions. In this paper, we investigate plane wave solutions of the Hermitian submonogenic system, leading to the construction of a Cauchy kernel. We also establish a Stokes type formula that, when applied to the Cauchy kernel provides an integral representation formula for Hermitian submonogenic functions.\vspace{0.2cm}\vspace{0.1cm}
\end{abstract}

\section{Introduction}

Let  $\mathbb{R}_{0,m}$ ($m\in\mathbb N$) be the real Clifford algebra generated by the imaginary units $e_1,\ldots,e_m$ with signature $(0,m)$. The elements $e_j$ can be identified with the Euclidean basis of the vector space $\mathbb R^m$ and a multiplication is introduced in $\mathbb{R}_{0,m}$ so that for any $\underline x=\sum_{j=1}^mx_je_j\in\mathbb R^m$:
\[\underline x^2=-\vert\underline x\vert^2=-\sum_{j=1}^mx_j^2.\] 
It thus follows that the elements $e_j$ must satisfy the following multiplication rules
\begin{alignat*}{2} 
e_j^2&=-1,&\qquad &j=1,\dots,m,\\
e_je_k+e_ke_j&=0,&\qquad &1\le j\neq k\le m.
\end{alignat*}
These relations determine the multiplication in $\mathbb{R}_{0,m}$ and a general element $a\in\mathbb R_{0,m}$ may be written as 
\[a=\sum_Aa_Ae_A,\quad a_A\in\mathbb R,\] 
in terms of the basis elements $e_A=e_{j_1}\dots e_{j_k}$, defined for every subset $A=\{j_1,\dots,j_k\}$ of $\{1,\dots,m\}$ with $j_1<\dots<j_k$ (for $A=\emptyset$ one puts $e_{\emptyset}=1$). So the dimension of $\mathbb R_{0,m}$ as a real linear space is $2^m$. Conjugation in $\mathbb R_{0,m}$ is given by $\overline a=\sum_Aa_A\overline e_A$, where $\overline e_A=\overline e_{j_k}\dots\overline e_{j_1}$ with $\overline e_j=-e_j$, $j=1,\dots,m$.

Elements of the form $a=\sum_Aa_Ae_A$ with $\vert A\vert=k$ are called $k$-vectors. Thus every element $a\in\mathbb R_{0,m}$ also admits a so-called multivector decomposition $a=\sum_{k=0}^m[a]_k$, where $[a]_k$ is the projection of $a$ on the space of $k$-vectors $\mathbb R_{0,m}^{(k)}$. Observe that $\mathbb R^{m+1}$ may be naturally embedded in the real Clifford algebra $\mathbb R_{0,m}$ by associating to any element $(x_0,x_1,\ldots,x_m)\in\mathbb R^{m+1}$ the paravector given by 
\[x_0+\underline x=x_0+\sum_{j=1}^mx_je_j\in\mathbb R_{0,m}^{(0)}\oplus\mathbb R_{0,m}^{(1)}.\]
A function $f:\Omega\rightarrow\mathbb{R}_{0,m}$ defined and continuously differentiable in an open set $\Omega$ in $\mathbb R^{m+1}$ (resp. $\mathbb R^m$), is said to be left monogenic (or simply monogenic) if
\[(\partial_{x_0}+\partial_{\underline x})f=0\quad(\text{resp.}\;\partial_{\underline x}f=0)\;\;\text{in}\;\;\Omega,\]
where $\partial_{\underline x}=\sum_{j=1}^me_j\partial_{x_j}$ is the Dirac operator in $\mathbb R^m$ (see e.g. \cite{BDS,DSS,GM,GuSp}). The differential operator $\partial_{x_0}+\partial_{\underline x}$, called generalized Cauchy-Riemann operator, gives a factorization of the Laplacian, i.e.
\[\Delta_{m+1}=\sum_{j=0}^m\partial_{x_j}^2=(\partial_{x_0}+\partial_{\underline x})(\partial_{x_0}-\partial_{\underline x}).\]
When allowing for complex coefficients, the multiplication rules imposed on the Euclidean basis $\{e_1,\ldots,e_m\}$ will generate the complex Clifford algebra $\mathbb C_m$. Since $\mathbb C_m$ can be seen as the complexification of the real Clifford algebra $\mathbb R_{0,m}$, i.e. $\mathbb C_m=\mathbb R_{0,m}\oplus i\mathbb R_{0,m}$, any complex Clifford number $c\in\mathbb C_m$ may be written as $c=a+ib$, $a,b\in\mathbb R_{0,m}$, leading to the definition of the Hermitian conjugation: $c^{\dagger}=\overline a-i\overline b$. This Hermitian conjugation gives rise to a norm on $\mathbb C_m$ given by 
\[\vert c\vert=\sqrt{[c^{\dagger}c]_0}.\]
The consideration of complexified Clifford algebras over even dimensional spaces leads to the construction of the so-called Witt basis of $\mathbb C_m$. Let $m=2n$, then the Witt basis elements are given by 
\[f_j=\frac{1}{2}(e_j-ie_{n+j}),\quad f_j^{\dagger}=-\frac{1}{2}(e_j+ie_{n+j}), \quad j=1,\dots,n.\] 
They satisfy the Grassmann identities $f_jf_k+f_kf_j=f_j^{\dagger}f_k^{\dagger}+f_k^{\dagger}f_j^{\dagger}=0$ and the duality identities $f_jf_k^{\dagger}+f_k^{\dagger}f_j=\delta_{jk}$, $j,k=1,\dots,n$. Rewriting the vector variable $\underline x$ and the Dirac operator $\partial_{\underline x}$ as 
\[\underline x=\sum_{j=1}^n(x_je_j+x_{n+j}e_{n+j}),\quad\partial_{\underline x}=\sum_{j=1}^n(e_j\partial_{x_j}+e_{n+j}\partial_{x_{n+j}})\]
and expressing them in terms of the Witt basis, they split in a natural way into $\underline x=\underline z-\underline z^{\dagger}$, $\partial_{\underline x}=2(-\partial_{\underline z}+\partial_{\underline z^{\dagger}})$ with 
\[\underline z=\sum_{j=1}^nz_jf_j,\quad\underline z^{\dagger}=(\underline z)^{\dagger}=\sum_{j=1}^n\overline z_jf_j^{\dagger},\quad\partial_{\underline z}=\sum_{j=1}^nf_j^{\dagger}\partial_{z_j},\quad\partial_{\underline z^{\dagger}}=(\partial_{\underline z})^{\dagger}=\sum_{j=1}^nf_j\partial_{\overline z_j},\]
where $n$ complex variables $z_j=x_j+ix_{n+j}$ have been introduced, with complex conjugates $\overline z_j=x_j-ix_{n+j}$, as well as the classical Cauchy-Riemann operators $\partial_{z_j}=\frac{1}{2}(\partial_{x_j}-i\partial_{x_{n+j}})$ and their complex conjugates $\partial_{\overline z_j}=\frac{1}{2}(\partial_{x_j}+i\partial_{x_{n+j}})$ in the respective complex $z_j$-planes, $j=1,\dots,n$. Note that
\begin{equation}\label{factDH}
\Delta_{2n}=\sum_{j=1}^{2n}\partial_{x_j}^2=4\left(\partial_{\underline z}\partial_{\underline z^{\dagger}}+\partial_{\underline z^{\dagger}}\partial_{\underline z}\right),
\end{equation}
which is the dual expression of 
\begin{equation}\label{dualexp}
\vert\underline z\vert^2=\vert\underline z^{\dagger}\vert^2=\underline z\,\underline z^{\dagger}+\underline z^{\dagger}\underline z.
\end{equation}
The above setting has been the starting point for the development of a new branch of Clifford analysis, so-called Hermitian Clifford analysis, giving rise to a function theory related to and even encompassing the several complex variables theory. The Hermitian monogenic ($h$-monogenic) functions are simultaneous null solutions of the operators $\partial_{\underline z}$ and $\partial_{\underline z^{\dagger}}$ (see e.g. \cite{H6A,H6A2,BHS,SaSo}), and in this way they are refining the properties of monogenic functions. They are invariant under the action of the unitary group $\text{U}(n)$.

Let us consider, for the sake of convenience, the $h$-monogenic system in dimension $m=2n+2$:
\begin{equation}\label{hms}
\left\{\begin{aligned}
(f_0^{\dagger}\partial_{z_0}+\partial_{\underline z})f&=0\\
(f_0\partial_{\overline{z}_0}+\partial_{\underline z^{\dagger}})f&=0
\end{aligned}\right.
\end{equation}
where $f:\Omega\subset\mathbb R^{2n+2}\rightarrow\mathbb{C}_{2n+2}$ is a continuously differentiable function and $z_0=x_0+iy_0$. For this system a Cauchy-Kowalevski extension problem was studied in \cite{BHLS}: given $g:\underline\Omega\subset\mathbb R^{2n}\rightarrow\mathbb{C}_{2n+2}$, find a solution $f$ of (\ref{hms}) such that $f\vert_{z_0=0}=g$. However, unlike the monogenic case, not every analytic function $g$ has a $h$-monogenic extension. Indeed, by multiplying the equations in (\ref{hms}) from the left by $f_0^{\dagger}f_0$ and $f_0f_0^{\dagger}$, we may verify that $f$ is a solution of the $h$-monogenic system (\ref{hms}) if and only if  
\begin{equation*}
\left\{\begin{aligned}
(f_0^{\dagger}\partial_{z_0}+f_0^{\dagger}f_0\partial_{\underline z})f&=0\\
f_0f_0^{\dagger}\partial_{\underline z}f&=0\\
(f_0\partial_{\overline z_0}+f_0f_0^{\dagger}\partial_{\underline z^{\dagger}})f&=0\\
f_0^{\dagger}f_0\partial_{\underline z^{\dagger}}f&=0,
\end{aligned}\right.
\end{equation*} 
and from the second and fourth equation of the last system we see that the Cauchy datum $g$ should satisfy the constrains $f_0f_0^{\dagger}\partial_{\underline z}g=f_0^{\dagger}f_0\partial_{\underline z^{\dagger}}g=0$. 

In order to overcome this difficulty we introduced in \cite{NDS1,NDS2} the following definition. 

\begin{defn}
A function $f:\,\Omega\subset\mathbb R^{2n+2}\rightarrow\mathbb{C}_{2n+2}$ is called $h$-submonogenic in $\Omega$ if and only if $f\in C^1(\Omega)$ and 
\begin{equation}\label{weakhms}
\left\{\begin{aligned}
(f_0^{\dagger}\partial_{z_0}+f_0^{\dagger}f_0\partial_{\underline z})f&=0\\
(f_0\partial_{\overline z_0}+f_0f_0^{\dagger}\partial_{\underline z^{\dagger}})f&=0
\end{aligned}\right.
\end{equation} 
or, equivalently, 
\[\mathbb Df=0,\;\text{where}\;\;\mathbb D=f_0\partial_{\overline z_0}+f_0^{\dagger}\partial_{z_0}+f_0f_0^{\dagger}\partial_{\underline z^{\dagger}}+f_0^{\dagger}f_0\partial_{\underline z}.\] 
\end{defn}
\noindent 
We note that $h$-submonogenic functions are invariant under the action of the unitary subgroup $\text{U}(n)$ of $\text{U}(n+1)$. It is remarkable that the Hermitian submonogenic system (\ref{weakhms}) can be written into a single equation, namely $\mathbb Df=0$. For this new system we studied a Cauchy-Kowalevski extension theorem showing that its solutions are determined by their Cauchy data and we solved the system explicitly for the Gaussian and for other special functions as well. In this way, we obtained Hermitian Bessel functions, Hermite polynomials and generalized powers. We also derived a Vekua-type system that describes all axially symmetric solutions. In this paper we aim at obtaining the fundamental solution of the Hermitian submonogenic system and to prove a Cauchy's integral formula for its solutions. 

The paper is organized as follows. In Section \ref{sect2} we recall some basic facts on special functions and study certain special integrals that will be needed in the course of the paper. In particular, we have to evaluate a number of integrals that arise from the application of Funk-Hecke's formula for spherical harmonics.

Section \ref{sect3} begins with a general study of the $h$-submonogenic system. Next we recall the Fourier expansion of the fundamental solution $E(x_0+\underline x)$ of the generalized Cauchy-Riemann operator $\partial_{x_0}+\partial_{\underline x}$ since this inspires the actual construction of the Cauchy kernel for the $h$-submonogenic system. To that end, we start with the study of plane wave exponential solutions $P$ that still depend on a vector parameter $\underline w$ and satisfy the $h$-submonogenic system from both sides: $\mathbb DP=0=P\mathbb D$. 

The Cauchy kernel is then obtained in Section \ref{sect4} by integrating these plane waves with respect to the vector parameter $\underline w$ and what we obtain is a solution to the two-sided $h$-submonogenic system $\mathbb Df=0=f\mathbb D$ with singularities on a half-line: $\{(x_0,\underline x)\in\mathbb R^{2n+1}:\;x_0\le0,\;\underline x=0\}$.

Finally, in Section \ref{sect5} we apply this kernel to arrive at a Cauchy integral representation formula for certain solutions of the $h$-submonogenic system. This result is obtained in three steps. First we establish a Stokes type theorem that couples left and right solutions of the $h$-submonogenic operator $\mathbb D$. Then we develop a Cauchy integral formula using the $x_0$-derivative of the Cauchy kernel; that kernel has a point singularity in the origin and it provides an integral formula for the $x_0$-derivative of a $h$-submonogenic function. By integrating over a half-line we obtain the desired Cauchy integral representation formula.

\section{Some preliminary results}\label{sect2}

In this section we collect a series of facts that will play an important role in the proof of our main result. We begin by recalling the well-known Beta function $\textrm{B}(x,y)$ defined by
\[\textrm{B}(x,y)=\int_0^1t^{x-1}(1-t)^{y-1}dt,\quad x,y>0\]
and its connection with the Gamma function $\Gamma$:
\[\textrm{B}(x,y)=\frac{\Gamma(x)\Gamma(y)}{\Gamma(x+y)}.\]
For positive half-integers, the values of $\Gamma$ are given by
\begin{equation}\label{gann2}
\Gamma\left(\frac{n}{2}\right)=\sqrt{\pi}\frac{(n-2)!!}{2^{(n-1)/2}},
\end{equation}
where $n!!$ denotes the double factorial of $n$. 

We shall also need the following Taylor series at $x=0$:

\begin{equation}\label{serieuno}
\frac{1}{(1-x)^m}=\sum_{k=0}^\infty\binom{k+m-1}{k}x^k,\quad\vert x\vert<1
\end{equation}

\begin{equation}\label{seriedos}
\frac{1}{(1+x)^{\frac{m+1}{2}}}=\sum_{k=0}^\infty\frac{(-1)^k(2k+m-1)!!}{(2k)!!(m-1)!!}\,x^k,\quad\vert x\vert<1.
\end{equation}

\begin{lem}\label{lemintindefS}
For $n\in\mathbb N$ we have
\begin{itemize}
\item[{\rm(i)}] $\displaystyle{\int_0^\infty x^ne^{\alpha x}dx=\frac{(-1)^{n+1}}{\alpha^{n+1}}\,n!},\quad\rm{Re}(\alpha)<0,$
\item[{\rm(ii)}] $\displaystyle{\int\frac{x^{2n-1}}{(1+x^2)^{\frac{2n+1}{2}}}dx=\frac{P_{2n-2}(x)}{(1+x^2)^{\frac{2n-1}{2}}}+C},\quad C\in\mathbb R,$
\item[{\rm(iii)}] $\displaystyle{\sum_{k=0}^\infty\frac{(-1)^k(2k+2n-1)!!}{(2k)!!(2k+2n)(2n-1)!!}\,x^{2k}=\frac{1}{x^{2n}}\left(\frac{P_{2n-2}(x)}{(1+x^2)^{\frac{2n-1}{2}}}+\frac{(2n-2)!!}{(2n-1)!!}\right)},\quad 0<\vert x\vert<1,$
\end{itemize}
where $\displaystyle{P_{2n-2}(x)=\sum_{j=0}^{n-1}a_jx^{2j}}$, $\quad\displaystyle{a_j=-\frac{2^{n-j-1}(n-1)!}{j!(2n-2j-1)!!}}$.
\end{lem}
\begin{proof}
The proof of (i) is straightforward. Indeed,
\[\int_0^\infty x^ne^{\alpha x}dx=\frac{e^{\alpha x}}{\alpha}\left(\sum_{k=0}^n(-1)^k\frac{n!x^{n-k}}{(n-k)!\alpha^k}\right)\bigg\vert_{x=0}^{x=\infty}=\frac{(-1)^{n+1}}{\alpha^{n+1}}\,n!.\]
Let $P_{2n-2}(x)=\sum_{j=0}^{n-1}a_jx^{2j}$ be an even polynomial of degree $2n-2$. Computing the integral in (ii) for lower values of $n$ we can infer that an antiderivative is given by
\[\frac{P_{2n-2}(x)}{(1+x^2)^{\frac{2n-1}{2}}}.\]
It is easy to verify  that 
\[(1+x^2)P_{2n-2}^\prime-(2n-1)xP_{2n-2}=x^{2n-1},\]
which implies that the coefficients of $P_{2n-2}$ should satisfy
\[2(j+1)a_{j+1}-(2n-2j-1)a_j=0,\quad j=0,\dots,n-2.\]
We can easily solve this recurrence relation to get $\displaystyle{a_{n-j}=-\frac{2^{j-1}(n-1)!}{(n-j)!(2j-1)!!}}$.

From (\ref{seriedos}) it follows that
\[\frac{t^{2n-1}}{(1+t^2)^{\frac{2n+1}{2}}}=\sum_{k=0}^\infty\frac{(-1)^k(2k+2n-1)!!}{(2k)!!(2n-1)!!}\,t^{2k+2n-1},\quad n\in\mathbb N.\]
Integrating on both sides leads to
\[\int_0^x\frac{t^{2n-1}}{(1+t^2)^{\frac{2n+1}{2}}}dt=\sum_{k=0}^\infty\frac{(-1)^k(2k+2n-1)!!}{(2k)!!(2k+2n)(2n-1)!!}\,x^{2k+2n},\]
and using (ii) we obtain
\begin{equation*}
\frac{1}{x^{2n}}\left(\frac{P_{2n-2}(x)}{(1+x^2)^{\frac{2n-1}{2}}}+\frac{(2n-2)!!}{(2n-1)!!}\right)=\sum_{k=0}^\infty\frac{(-1)^k(2k+2n-1)!!}{(2k)!!(2k+2n)(2n-1)!!}\,x^{2k}
\end{equation*}
as desired.
\end{proof}

\begin{lem}\label{tresintegrales}
Let $r=\vert\underline x\vert$. If $x_0>r\ne0$, then
\begin{itemize}
\item[{\rm(i)}] $\displaystyle{\int_{-1}^{1}\frac{(1-t^2)^{\frac{m-3}{2}}}{(x_0-irt)^m}dt=\frac{\sqrt{2\pi}(m-3)!!x_0}{(m-2)!!(x_0^2+r^2)^{\frac{m+1}{2}}}}$
\item[{\rm(ii)}] $\displaystyle{\int_{-1}^{1}\frac{t(1-t^2)^{\frac{m-3}{2}}}{(x_0-irt)^m}dt=\frac{i\sqrt{2\pi}(m-3)!!r}{(m-2)!!(x_0^2+r^2)^{\frac{m+1}{2}}}}$ 
\item[{\rm(iii)}] $\displaystyle{\int_{-1}^{1}\frac{t^2(1-t^2)^{\frac{2n-3}{2}}}{(x_0-irt)^{2n}}dt=\frac{\sqrt{2\pi}x_0\sum_{j=0}^nb_jx_0^{2n-2j}r^{2j}}{(2n-2)!!r^{2n}(x_0^2+r^2)^\frac{2n+1}{2}}-\frac{\sqrt{2\pi}}{r^{2n}}},$ 
\end{itemize}
where $\displaystyle{b_n=2n(2n-3)!!,\quad b_j=\frac{(2n-2)!!(2n+1)!!}{(2j)!!(2n-2j+1)!!},\quad j=0,\dots,n-1}$.
\end{lem}
\begin{proof}
We only prove identities (i) and (iii). The proof of (ii) is similar to that of (i). Using (\ref{serieuno}) we obtain 
\[\int_{-1}^{1}\frac{(1-t^2)^{\frac{m-3}{2}}}{(x_0-irt)^m}dt=\frac{1}{x_0^m}\int_{-1}^{1}\frac{(1-t^2)^{\frac{m-3}{2}}}{\left(1-\frac{ir}{x_0}t\right)^m}dt=\frac{1}{x_0^m}\sum_{k=0}^\infty\binom{k+m-1}{k}\left(\frac{ir}{x_0}\right)^kI_k,\]
where $I_k=\displaystyle{\int_{-1}^{1}t^k(1-t^2)^{\frac{m-3}{2}}dt}$. Taking into account the parity of the integrand, we have
\[I_{2k}=2\int_{0}^{1}t^{2k}(1-t^2)^{\frac{m-3}{2}}dt,\quad I_{2k+1}=0.\]
Making the change of variables $x=t^2$ leads to
\begin{align*}
I_{2k}&=\int_{0}^{1}x^{k-\frac{1}{2}}(1-x)^{\frac{m-3}{2}}dx=\textrm{B}\left(k+\frac{1}{2},\frac{m-1}{2}\right)\\
&=\frac{\Gamma\left(k+\frac{1}{2}\right)\Gamma\left(\frac{m-1}{2}\right)}{\Gamma\left(k+\frac{m}{2}\right)}=\sqrt{2\pi}\,\frac{(2k-1)!!(m-3)!!}{(2k+m-2)!!},
\end{align*}
where we have also used (\ref{gann2}). By the above relations and using (\ref{seriedos}) we get
\begin{align*}
\int_{-1}^{1}\frac{(1-t^2)^{\frac{m-3}{2}}}{(x_0-irt)^m}dt&=\frac{\sqrt{2\pi}(m-3)!!}{(m-2)!!x_0^m}\sum_{k=0}^\infty\frac{(-1)^k(2k+m-1)!!}{(2k)!!(m-1)!!}\left(\frac{r}{x_0}\right)^{2k}\\
&=\frac{\sqrt{2\pi}(m-3)!!}{(m-2)!!x_0^m\left(1+\left(\frac{r}{x_0}\right)^2\right)^{\frac{m+1}{2}}},
\end{align*}
from which the first identity follows.

Using similar arguments we also get that
\[\int_{-1}^{1}\frac{t^2(1-t^2)^{\frac{2n-3}{2}}}{(x_0-irt)^{2n}}dt=\frac{1}{x_0^{2n}}\sum_{k=0}^\infty\binom{2k+2n-1}{2k}\left(\frac{ir}{x_0}\right)^{2k}Q_{k},\]
where
\begin{align*}
Q_k=\int_{-1}^{1}t^{2k+2}(1-t^2)^{\frac{2n-3}{2}}dt&=\textrm{B}\left(k+\frac{3}{2},\frac{2n-1}{2}\right)\\
&=\sqrt{2\pi}\,\frac{(2k+1)!!(2n-3)!!}{(2k+2n)!!}.
\end{align*}
Therefore
\[\int_{-1}^{1}\frac{t^2(1-t^2)^{\frac{2n-3}{2}}}{(x_0-irt)^{2n}}dt=\frac{\sqrt{2\pi}}{(2n-1)(2n-2)!!x_0^{2n}}\sum_{k=0}^\infty\frac{(-1)^k(2k+1)(2k+2n-1)!!}{(2k)!!(2k+2n)}\left(\frac{r}{x_0}\right)^{2k}\]
\[\qquad\qquad=\frac{\sqrt{2\pi}(2n-3)!!}{(2n-2)!!x_0^{2n}}\sum_{k=0}^\infty\frac{(-1)^k(2k+2n-1)!!}{(2k)!!(2n-1)!!}\left(\frac{r}{x_0}\right)^{2k}\]
\[\qquad\qquad\qquad-\frac{\sqrt{2\pi}(2n-1)!!}{(2n-2)!!x_0^{2n}}\sum_{k=0}^\infty\frac{(-1)^k(2k+2n-1)!!}{(2k)!!(2k+2n)(2n-1)!!}\left(\frac{r}{x_0}\right)^{2k}.\]
Using (\ref{seriedos}) and (iii) of Lemma \ref{lemintindefS} we obtain
\[\int_{-1}^{1}\frac{t^2(1-t^2)^{\frac{2n-3}{2}}}{(x_0-irt)^{2n}}dt=\frac{\sqrt{2\pi}(2n-3)!!}{(2n-2)!!x_0^{2n}\left(1+\left(\frac{r}{x_0}\right)^2\right)^{\frac{2n+1}{2}}}\] 
\[-\frac{\sqrt{2\pi}(2n-1)!!}{(2n-2)!!r^{2n}}\left(\frac{P_{2n-2}\left(\frac{r}{x_0}\right)}{\left(1+\left(\frac{r}{x_0}\right)^2\right)^{\frac{2n-1}{2}}}+\frac{(2n-2)!!}{(2n-1)!!}\right)=\frac{\sqrt{2\pi}x_0\sum_{j=0}^nb_jx_0^{2n-2j}r^{2j}}{(2n-2)!!r^{2n}(x_0^2+r^2)^\frac{2n+1}{2}}-\frac{\sqrt{2\pi}}{r^{2n}},\]
which is the desired conclusion.
\end{proof}

\begin{thm}[Funk-Hecke's formula \cite{Hoch}]
Suppose that $\displaystyle{\int_{-1}^1\vert F(t)\vert(1-t^2)^{(p-3)/2}dt<\infty}$ and let $\underline\xi\in S^{p-1}$. If $Y_k(\underline x)$ is a spherical harmonic of degree $k$ in $\mathbb R^p$, then 
\[\int_{S^{p-1}}F(\langle\underline\xi,\underline\eta\rangle)Y_k(\underline\eta)dS(\underline\eta)=\sigma_{p-1}C_k(1)^{-1}Y_{k}(\underline\xi)\int_{-1}^1F(t)C_k(t)(1-t^2)^{(p-3)/2}dt,\]
where $C_k(t)$ denotes the Gegenbauer polynomial $C^{\lambda}_k(t)$ with $\lambda=(p-2)/2$ and $\sigma_{p-1}=\frac{2\pi^{\frac{p-1}{2}}}{\Gamma\left(\frac{p-1}{2}\right)}$ is the surface area of the unit sphere $S^{p-2}$ in $\mathbb R^{p-1}$.
\end{thm}
\noindent
We recall that the  Gegenbauer polynomial $C^{\lambda}_k(t)$ are orthogonal polynomials on the interval $[-1,1]$ with respect to the weight function $(1-t^2)^{\lambda-1/2}$ and satisfy the recurrence relation
\begin{align*}
C^{\lambda}_k(t)&=\frac{1}{k}\left(2(k+\lambda-1)t\,C^{\lambda}_{k-1}(t)-(k+2\lambda-2)C^{\lambda}_{k-2}(t)\right),\quad k\ge2,\\
C^{\lambda}_0(t)&=1,\qquad C^{\lambda}_1(t)=2\lambda t.
\end{align*}
Let $\mathsf{P}(k)$ be the set of all homogeneous polynomials of degree $k$ defined in $\mathbb R^p$. By $\mathsf{H}(k)$ we denote the polynomials in $\mathsf{P}(k)$ which are harmonic.

\begin{thm}[Fischer decomposition \cite{DSS}]\label{Fischer}
If $P_k(\underline x)\in\mathsf{P}(k)$, with $k\ge2$, then there exist unique polynomials $H_k(\underline x)\in\mathsf{H}(k)$ and $P_{k-2}(\underline x)\in\mathsf{P}(k-2)$ such that 
\[P_k(\underline x)=H_k(\underline x)+\vert\underline x\vert^2P_{k-2}(\underline x).\]
\end{thm}
\noindent
Finally, it is useful to recall the following well-known formula for computing multiple integrals in polar coordinates:
\begin{equation}\label{fintesfecorde}
\int_{\mathbb R^p}f(\underline x)dV(\underline x)=\int_0^\infty r^{p-1}\left(\int_{S^{p-1}} f(r\underline\omega)dS(\underline\omega)\right)dr.
\end{equation}

\section{Plane wave $h$-submonogenic functions}\label{sect3}

We begin by observing that any $\mathbb{C}_{2n+2}$-valued function $f$ may be uniquely written into the form
\[f=A+f_0B+f_0^{\dagger}C+f_0^{\dagger}f_0D,\]
where $A$, $B$, $C$, $D$ are $\mathbb{C}_{2n}$-valued functions. A direct computation then yields
\begin{align*}
(f_0^{\dagger}\partial_{z_0}+f_0^{\dagger}f_0\partial_{\underline z})f&=f_0^{\dagger}(\partial_{z_0}A-\partial_{\underline z}C)+f_0^{\dagger}f_0\big(\partial_{z_0}B+\partial_{\underline z}(A+D)\big),\\
(f_0\partial_{\overline z_0}+f_0f_0^{\dagger}\partial_{\underline z^{\dagger}})f&=f_0\big(\partial_{\overline z_0}(A+D)-\partial_{\underline z^{\dagger}}B\big)+f_0f_0^{\dagger}(\partial_{\underline z^{\dagger}}A+\partial_{\overline z_0}C).
\end{align*}
Therefore the $h$-submonogenic system (\ref{weakhms}) is equivalent to the following systems of equations
\begin{equation}\label{whmFs1}
\left\{\begin{aligned}
\partial_{z_0}A&=\partial_{\underline z}C\\
\partial_{\underline z^{\dagger}}A&=-\partial_{\overline z_0}C
\end{aligned}\right.
\end{equation}
\begin{equation}\label{whmFs2}
\left\{\begin{aligned}
\partial_{z_0}B&=-\partial_{\underline z}(A+D)\\
\partial_{\underline z^{\dagger}}B&=\partial_{\overline z_0}(A+D).
\end{aligned}\right.
\end{equation}

\begin{rem}
We may also consider the dual equation $f\mathbb D=0$ and if we now rewrite function $f$ as   
\[f=\mathsf{A}+\mathsf{B}f_0+\mathsf{C}f_0^{\dagger}+\mathsf{D}f_0^{\dagger}f_0,\]
it follows that $f\mathbb D=0$ is fulfilled if and only if
\begin{equation*}
\left\{\begin{aligned}
\partial_{\overline z_0}\mathsf{A}&=\mathsf{B}\partial_{\underline z}\\
\mathsf{A}\partial_{\underline z^{\dagger}}&=-\partial_{z_0}\mathsf{B}
\end{aligned}\right.
\end{equation*}
\begin{equation*}
\left\{\begin{aligned}
\partial_{\overline z_0}\mathsf{C}&=-(\mathsf{A}+\mathsf{D})\partial_{\underline z}\\
\mathsf{C}\partial_{\underline z^{\dagger}}&=\partial_{z_0}(\mathsf{A}+\mathsf{D}).
\end{aligned}\right.
\end{equation*}
\end{rem}

\begin{prop}
Every $h$-submonogenic function $f$ satisfies the equation  
\begin{equation*}
\Delta_{2}(\Delta_{2}+\Delta_{2n})f=0,
\end{equation*}
where $\Delta_2=\partial_{x_0}^2+\partial_{y_0}^2$. 
\end{prop}
\begin{proof}
From (\ref{whmFs1}) we get
\begin{align*}
\partial_{z_0}\partial_{\overline z_0}A&=\partial_{\overline z_0}\partial_{\underline z}C=-\partial_{\underline z}\partial_{\underline z^{\dagger}}A\\
\partial_{z_0}\partial_{\overline z_0}C&=-\partial_{z_0}\partial_{\underline z^{\dagger}}A=-\partial_{\underline z^{\dagger}}\partial_{\underline z}C.
\end{align*}
So that
\begin{equation*}
\begin{aligned}
\partial_{z_0}\partial_{\overline z_0}A+\partial_{\underline z}\partial_{\underline z^{\dagger}}A&=0\\
\partial_{z_0}\partial_{\overline z_0}C+\partial_{\underline z^{\dagger}}\partial_{\underline z}C&=0.
\end{aligned}
\end{equation*}
In a similar fashion we get from (\ref{whmFs2}) that 
\begin{equation*}
\begin{aligned}
\partial_{z_0}\partial_{\overline z_0}B+\partial_{\underline z}\partial_{\underline z^{\dagger}}B&=0\\
\partial_{z_0}\partial_{\overline z_0}(A+D)+\partial_{\underline z^{\dagger}}\partial_{\underline z}(A+D)&=0.
\end{aligned}
\end{equation*}
Moreover using (\ref{factDH}) we obtain
\[\left(\frac{1}{4}\Delta_{2}+\partial_{\underline z}\partial_{\underline z^{\dagger}}\right)\left(\frac{1}{4}\Delta_{2}+\partial_{\underline z^{\dagger}}\partial_{\underline z}\right)=\frac{1}{16}\left(\Delta_{2}^2+\Delta_{2}\Delta_{2n}\right)=\frac{1}{16}\Delta_{2}(\Delta_{2}+\Delta_{2n}),\]
from which the proposition easily follows. 
\end{proof}
\noindent
The fundamental solution of the generalized Cauchy-Riemann operator $\partial_{x_0}+\partial_{\underline x}$ is given by 
\[E(x_0+\underline x)=\frac{1}{\sigma_{m+1}}\frac{x_0-\underline x}{\vert x_0+\underline x\vert^{m+1}},\]
where $\sigma_{m+1}$ denotes the surface area of the unit sphere $S^m$ in $\mathbb R^{m+1}$. 

This function is two-sided monogenic, i.e. $(\partial_{x_0}+\partial_{\underline x})E=0=E(\partial_{x_0}+\partial_{\underline x})$. This fact plays a key role in the proof of Cauchy's integral formula for monogenic functions. Moreover, $E$ satisfies the following plane wave representation (see \cite{CISSS,Som}):

\begin{thm}\label{PWCKernelR}
For $\vert x_0\vert>\vert \underline x\vert\ne0$ it holds:
\[2^{m+1}\pi^{m-1}\sqrt{2\pi}\sgn(x_0)E(x_0+\underline x)=\int_{\mathbb R^m}\displaystyle{e^{i\langle\underline x, \underline u\rangle-\vert x_0\vert\vert\underline u\vert}}\left(1+i\sgn(x_0)\frac{\underline u}{\vert\underline u\vert}\right)dV(\underline u),\]
where $\sgn(x_0)=x_0/\vert x_0\vert$.
\end{thm}

We are going to use these two fundamental properties of $E$ in order to obtain the fundamental solution of the $h$-submonogenic system. We shall begin by constructing in our setting the analogue of the plane wave monogenic function 
\begin{equation}\label{exppwave}
e^{i\langle\underline x, \underline u\rangle-x_0\vert\underline u\vert}\left(1+i\frac{\underline u}{\vert\underline u\vert}\right),
\end{equation}
which has to be a solution of $\mathbb Df=0=f\mathbb D$, and then compute the corresponding integral in $\mathbb R^{2n}$. 

\begin{prop}\label{subhmpew}
Let $\lambda,\mu\in\mathbb C\setminus\{0\}$ and suppose that $\underline w=\sum_{j=1}^nw_jf_j=\sum_{j=1}^n(u_j+iu_{n+j})f_j$ is a fixed non-zero Hermitian vector. If $\alpha_1$ and $\alpha_2$ are constants satisfying $\vert\underline w\vert^2\alpha_1\alpha_2=-\lambda\mu$, then the function 
\[e^{\alpha_1\theta+\alpha_2\overline\theta+\lambda z_0+\mu\overline z_0}\left(f_0f_0^{\dagger}\frac{\underline w^{\dagger}\underline w}{\vert\underline w\vert^2}-\frac{\alpha_2}{\mu}f_0^{\dagger}\underline w\right),\quad\theta=\{\underline z,\underline w^{\dagger}\}=\sum_{j=1}^nz_j\overline w_j\]
is $h$-submonogenic in $\mathbb R^{2n+2}$. 
\end{prop}
\begin{proof}
We look for special solutions of the systems (\ref{whmFs1}) and (\ref{whmFs2}), which we assume to be of the form
\begin{align*}
A&=e^{\lambda z_0+\mu\overline z_0}\underline w^{\dagger}\underline wa,&
B&=e^{\lambda z_0+\mu\overline z_0}\underline w^{\dagger}b,\\
C&=e^{\lambda z_0+\mu\overline z_0}\underline wc,&
D&=e^{\lambda z_0+\mu\overline z_0}\underline w^{\dagger}\underline wd,
\end{align*}
where $a$, $b$, $c$, $d$ are $\mathbb C$-valued continuously differentiable functions depending on the two variables $(\theta,\overline\theta)$.

It is easily seen that
\[\partial_{z_0}A=\lambda e^{\lambda z_0+\mu\overline z_0}\underline w^{\dagger}\underline wa,\quad\partial_{\overline z_0}C=\mu e^{\lambda z_0+\mu\overline z_0}\underline wc,\]
and using the following equalities
\[\partial_{\underline z^{\dagger}}a=\sum_{j=1}^nf_j\partial_{\overline z_j}a=\sum_{j=1}^nf_j(\partial_{\overline\theta}a)(\partial_{\overline z_j}\overline\theta)=\underline w\,\partial_{\overline\theta}a\]
\[\partial_{\underline z}c=\sum_{j=1}^nf_j^\dagger\partial_{z_j}c=\sum_{j=1}^nf_j^\dagger(\partial_{\theta}c)(\partial_{z_j}\theta)=\underline w^{\dagger}\partial_{\theta}c,\]
we also obtain
\begin{align*}
\partial_{\underline z}C&=e^{\lambda z_0+\mu\overline z_0}\underline w^{\dagger}\underline w\,\partial_{\theta}c\\
\partial_{\underline z^{\dagger}}A&=e^{\lambda z_0+\mu\overline z_0}\underline w\,\underline w^{\dagger}\underline w\,\partial_{\overline\theta}a=e^{\lambda z_0+\mu\overline z_0}\underline w\vert\underline w\vert^2\partial_{\overline\theta}a.
\end{align*}
Therefore, system (\ref{whmFs1}) takes the form
\begin{equation}\label{sistac}
\left\{\begin{aligned}
\partial_{\theta}c&=\lambda a\\
\vert\underline w\vert^2\partial_{\overline\theta}a&=-\mu c.
\end{aligned}\right.
\end{equation}
In a similar way we have that
\begin{align*}
&\partial_{z_0}B=\lambda e^{\lambda z_0+\mu\overline z_0}\underline w^{\dagger}b,\quad\partial_{\overline z_0}(A+D)=\mu e^{\lambda z_0+\mu\overline z_0}\underline w^{\dagger}\underline w(a+d)\\
&\partial_{\underline z}(A+D)=e^{\lambda z_0+\mu\overline z_0}\big(\underline w^{\dagger}\big)^2\underline w\,\partial_{\theta}(a+d)=0.
\end{align*}
We then get from (\ref{whmFs2}) that
\begin{equation*}
b=0,\qquad a+d=0.
\end{equation*}
System (\ref{sistac}) can easily be solved by setting $a=e^{\alpha_1\theta+\alpha_2\overline\theta}$, $\alpha_1,\alpha_2\in\mathbb C$. Indeed,  it follows at once that  
\[c=-\frac{\alpha_2}{\mu}\vert\underline w\vert^2e^{\alpha_1\theta+\alpha_2\overline\theta},\quad\vert\underline w\vert^2\alpha_1\alpha_2=-\lambda\mu,\]
completing the proof.
\end{proof}
\noindent
In this paper we shall deal with a particular case of this function. Namely, we will consider
\begin{equation}\label{expwaveh}
e^{i\sum_{j=1}^n(x_{n+j}u_j-x_ju_{n+j})-x_0\vert\underline w\vert}\left(f_0f_0^{\dagger}\frac{\underline w^{\dagger}\underline w}{\vert\underline w\vert^2}-f_0^{\dagger}\frac{\underline w}{\vert\underline w\vert}\right),
\end{equation}
which corresponds to the case $\lambda=\mu=-\frac{1}{2}\vert\underline w\vert$, $\alpha_1=\frac{1}{2}$, $\alpha_2=-\frac{1}{2}$. Other choices of $\lambda,\mu,\alpha_1,\alpha_2$ would result in plane waves that depend on both $x_0$ and $y_0$, but those plane waves would be similar to the above special case (which is clear from the restrictions on $\lambda,\mu,\alpha_1,\alpha_2$).

Looking closely at (\ref{expwaveh}) we are able to devise a simple method to generate $h$-submonogenic functions not depending on the $y_0$ variable. 

\begin{prop}
Assume that $h(x+iy)$ is an anti-holomorphic function. If we put $x=x_0\vert\underline w\vert$ and $y=\sum_{j=1}^n(x_{n+j}u_j-x_ju_{n+j})$, then the function  
\[h(x+iy)\left(f_0f_0^{\dagger}\frac{\underline w^{\dagger}\underline w}{\vert\underline w\vert^2}-f_0^{\dagger}\frac{\underline w}{\vert\underline w\vert}\right)\]
is $h$-submonogenic. 
\end{prop}
\begin{proof}
Straightforward calculations yield
\begin{align*}
2(f_0^{\dagger}\partial_{z_0}+f_0^{\dagger}f_0\partial_{\underline z})h&=f_0^{\dagger}\vert\underline w\vert\partial_xh-if_0^{\dagger}f_0(\partial_yh)\underline w^{\dagger}\\
&=\vert\underline w\vert \partial_xh\left(f_0^{\dagger}-f_0^{\dagger}f_0\frac{\underline w^{\dagger}}{\vert\underline w\vert}\right).
\end{align*}
Similarly, we obtain
\[2(f_0\partial_{\overline z_0}+f_0f_0^{\dagger}\partial_{\underline z^{\dagger}})h=\vert\underline w\vert \partial_xh\left(f_0+f_0f_0^{\dagger}\frac{\underline w}{\vert\underline w\vert}\right).\]
The proof now follows easily using the fact that $f_0f_0^{\dagger}\frac{\underline w^{\dagger}\underline w}{\vert\underline w\vert^2}-f_0^{\dagger}\frac{\underline w}{\vert\underline w\vert}$ is a zero divisor.
\end{proof}

Although (\ref{expwaveh}) seems to be an excellent candidate for playing the role of the plane wave function (\ref{exppwave}), it does not satisfy $f\mathbb D=0$. We solve this matter by multiplying (\ref{expwaveh}) from the right by $f_0-\underline w^{\dagger}/\vert\underline w\vert$. 

\begin{thm}
The function $P$ defined by 
\begin{equation}\label{realexpwh}
P(x_0,\underline x)=e^{i\sum_{j=1}^n(x_{n+j}u_j-x_ju_{n+j})-x_0\vert\underline w\vert}\left(f_0f_0^{\dagger}\frac{\underline w^{\dagger}\underline w}{\vert\underline w\vert^2}-f_0^{\dagger}\frac{\underline w}{\vert\underline w\vert}\right)\left(f_0-\frac{\underline w^{\dagger}}{\vert\underline w\vert}\right)
\end{equation}
satisfies the equations $\mathbb DP=0=P\mathbb D$.
\end{thm}
\begin{proof}
It is clear from Proposition \ref{subhmpew} that $\mathbb DP=0$. Moreover, note that 
\begin{multline*}
\mathbb D\,e^{i\sum_{j=1}^n(x_{n+j}u_j-x_ju_{n+j})-x_0\vert\underline w\vert}\\
=-\frac{1}{2}e^{i\sum_{j=1}^n(x_{n+j}u_j-x_ju_{n+j})-x_0\vert\underline w\vert}(f_0\vert\underline w\vert+f_0^{\dagger}\vert\underline w\vert+f_0f_0^{\dagger}\underline w-f_0^{\dagger}f_0\underline w^{\dagger}),
\end{multline*}
and 
\[\left(f_0f_0^{\dagger}\frac{\underline w^{\dagger}\underline w}{\vert\underline w\vert^2}-f_0^{\dagger}\frac{\underline w}{\vert\underline w\vert}\right)\left(f_0-\frac{\underline w^{\dagger}}{\vert\underline w\vert}\right)(f_0\vert\underline w\vert+f_0^{\dagger}\vert\underline w\vert+f_0f_0^{\dagger}\underline w-f_0^{\dagger}f_0\underline w^{\dagger})\]
\[=\left(f_0f_0^{\dagger}\frac{\underline w^{\dagger}\underline w}{\vert\underline w\vert^2}-f_0^{\dagger}\frac{\underline w}{\vert\underline w\vert}\right)\left(f_0^{\dagger}\underline w^{\dagger}+f_0f_0^{\dagger}\frac{\underline w\,\underline w^{\dagger}}{\vert\underline w\vert}\right)=0.\] 
From these equalities it easily follows that $P$ also satisfies $P\mathbb D=0$. 
\end{proof}

\section{A Cauchy kernel with a half-line of singularities}\label{sect4}

Due to the nature of the plane waves (\ref{expwaveh}) we want to consider solutions of the system (\ref{weakhms}) not depending on the $y_0$ variable. For this case we try to obtain the fundamental solution by computing the integral
\begin{equation}\label{intfundsol}
I(x_0,\underline x)=\int_{\mathbb R^{2n}}PdV(\underline w),\quad x_0>0,
\end{equation}
where $P$ denotes the function defined in (\ref{realexpwh}). Clearly, to perform this task we must compute integrals of the form
\[\int_{\mathbb R^{2n}}e^{i\sum_{j=1}^n(x_{n+j}u_j-x_ju_{n+j})-x_0\vert\underline w\vert}F(\underline w,\underline w^{\dagger})dV(\underline w),\]
where $F(\underline w,\underline w^{\dagger})$ can be equal to $1$, $\displaystyle{\frac{\underline w}{\vert\underline w\vert}}$, $\displaystyle{\frac{\underline w^{\dagger}}{\vert\underline w\vert}}$, $\displaystyle{\frac{\underline w^{\dagger}\underline w}{\vert\underline w\vert^2}}$ or $\displaystyle{\frac{\underline w\,\underline w^{\dagger}}{\vert\underline w\vert^2}}$. These integrals shall be denoted by $I_0,\dots,I_4$, respectively. 

Define
\begin{align*}
G(x_0,r)&=\frac{x_0\sum_{j=0}^nc_jx_0^{2n-2j}r^{2j}}{r^{2n}(x_0^2+r^2)^\frac{2n+1}{2}}-\frac{(2n)!!}{(2n-1)!!r^{2n}},\\
H(x_0,r)&=\frac{x_0\sum_{j=0}^nd_jx_0^{2n-2j}r^{2j}}{r^{2n}(x_0^2+r^2)^\frac{2n+1}{2}}-\frac{2(2n-2)!!}{(2n-1)!!r^{2n}},\quad r=\vert \underline x\vert,
\end{align*}
where
\begin{align*}
c_j&=\frac{(2n+1)(2n)!!}{(2j)!!(2n-2j+1)!!},\;\;j=0,\dots,n\\
d_n&=2,\quad d_j=\frac{c_j}{n},\;\;j=0,\dots,n-1,
\end{align*}
and put 
\[\lambda_n=\sqrt{2\pi}(2n-1)!!(2n-3)!!\sigma_{2n-1}.\] 
Our next result constitutes an analogue of Theorem \ref{PWCKernelR}. 

\begin{thm}
Let $\beta=\sum_{j=1}^nf_j^\dagger f_j$. If $x_0>r\ne0$, then integral (\ref{intfundsol}) equals
\begin{align*}
I&=f_0^{\dagger}f_0I_1-f_0f_0^{\dagger}I_2+f_0I_3+f_0^{\dagger}I_4\\
&=-I_2+f_0I_3+f_0^{\dagger}I_4+f_0^{\dagger}f_0(I_1+I_2),
\end{align*}
where
\[I_1=\frac{\lambda_n\underline z}{(x_0^2+r^2)^{\frac{2n+1}{2}}},\quad I_2=-\frac{\lambda_n\underline z^{\dagger}}{(x_0^2+r^2)^{\frac{2n+1}{2}}},\]
\[I_3=\lambda_n\left(\frac{\underline z^\dagger\underline z}{r^2}G-\beta H\right),\quad I_4=\lambda_n\left(\frac{x_0}{(x_0^2+r^2)^{\frac{2n+1}{2}}}-\frac{\underline z^\dagger\underline z}{r^2}G+\beta H\right).\]
\end{thm}
\begin{proof}
It is easily seen from (\ref{intfundsol}) that $I=f_0^{\dagger}f_0I_1-f_0f_0^{\dagger}I_2+f_0I_3+f_0^{\dagger}I_4$. Now, using (\ref{fintesfecorde}) and writing the vector variable $\underline x$ as
\[\underline x=\sum_{j=1}^n(x_je_j+x_{n+j}e_{n+j})=r\sum_{j=1}^n(\omega_je_j+\omega_{n+j}e_{n+j}),\quad r=\vert\underline x\vert,\]
we obtain
\[I_0=\int_0^\infty\rho^{2n-1}\left(\int_{S^{2n-1}}e^{ir\rho\sum_{j=1}^n(\omega_{n+j}\nu_j-\omega_j\nu_{n+j})-x_0\rho}dS(\underline\xi)\right)d\rho\]
with $\underline\xi=\sum_{j=1}^n\xi_jf_j=\sum_{j=1}^n(\nu_j+i\nu_{n+j})f_j$. By Funk-Hecke's formula, it follows that 
\[I_0=\sigma_{2n-1}\int_0^\infty\rho^{2n-1}\left(\int_{-1}^1e^{(-x_0+irt)\rho}(1-t^2)^{(2n-3)/2}dt\right)d\rho.\]
Changing the order of integration and using (i) of Lemma \ref{lemintindefS}, we obtain
\[I_0=(2n-1)!\sigma_{2n-1}\int_{-1}^1\frac{(1-t^2)^{\frac{2n-3}{2}}}{(x_0-irt)^{2n}}dt.\]
Equality (i) of Lemma \ref{tresintegrales} now implies
\[I_0=\frac{\lambda_nx_0}{(x_0^2+r^2)^{\frac{2n+1}{2}}}.\]
In a similar fashion, using now equality (ii) of Lemma \ref{tresintegrales}, it follows that 
\begin{align*}
I_1&=\int_0^\infty\rho^{2n-1}\left(\int_{S^{2n-1}}e^{ir\rho\sum_{j=1}^n(\omega_{n+j}\nu_j-\omega_j\nu_{n+j})-x_0\rho}\underline\xi\,dS(\underline\xi)\right)d\rho\\
&=-i\sigma_{2n-1}\frac{\underline z}{r}\int_0^\infty\rho^{2n-1}\left(\int_{-1}^1e^{(-x_0+irt)\rho}t(1-t^2)^{(2n-3)/2}dt\right)d\rho\\
&=-i(2n-1)!\sigma_{2n-1}\frac{\underline z}{r}\int_{-1}^1\frac{t(1-t^2)^{\frac{2n-3}{2}}}{(x_0-irt)^{2n}}dt=\frac{\lambda_n\underline z}{(x_0^2+r^2)^{\frac{2n+1}{2}}}.
\end{align*}
After these reasonings one can easily obtain that 
\[I_2=-\frac{\lambda_n\underline z^{\dagger}}{(x_0^2+r^2)^{\frac{2n+1}{2}}}.\]
The computation of $I_3$ is more involved. Formula (\ref{fintesfecorde}) gives
\[I_3=\int_0^\infty\rho^{2n-1}\left(\int_{S^{2n-1}}e^{ir\rho\sum_{j=1}^n(\omega_{n+j}\nu_j-\omega_j\nu_{n+j})-x_0\rho}\underline\xi^{\dagger}\underline\xi\,dS(\underline\xi)\right)d\rho.\]
The main difficulty here is that we may not apply directly Funk-Hecke's formula since $\underline\xi^{\dagger}\underline\xi$ is not spherical harmonic. We first need to write $\underline\xi^{\dagger}\underline\xi$ as  
\[\underline\xi^{\dagger}\underline\xi=\sum_{j=1}^nf_j^{\dagger}f_j\left\vert\xi_j\right\vert^2+\sum_{\substack{j,k=1\\j\ne k}}^nf_j^{\dagger}f_k\,\overline\xi_j\xi_k\]
and to apply the Fischer decomposition to $\vert\xi_j\vert^2$: 
\[\vert\xi_j\vert^2=\left(\vert\xi_j\vert^2-\frac{1}{n}\left\vert\underline\xi\right\vert^2\right)+\frac{1}{n}\left\vert\underline\xi\right\vert^2.\]
It thus follows that 
\begin{multline*}
\int_{S^{2n-1}}e^{ir\rho\sum_{j=1}^n(\omega_{n+j}\nu_j-\omega_j\nu_{n+j})}\vert\xi_j\vert^2dS(\underline\xi)=\frac{\sigma_{2n-1}}{n}\int_{-1}^1e^{ir\rho t}(1-t^2)^{(2n-3)/2}dt\\
+\frac{\sigma_{2n-1}}{(2n-1)}\left(\frac{\vert z_j\vert^2}{r^2}-\frac{1}{n}\right)\int_{-1}^1e^{ir\rho t}(2nt^2-1)(1-t^2)^{(2n-3)/2}dt.
\end{multline*}
Changing the order of integration and using again (i) of Lemma \ref{lemintindefS}, we thus get
\begin{multline*}
\int_0^\infty\rho^{2n-1}\left(\int_{S^{2n-1}}e^{ir\rho\sum_{j=1}^n(\omega_{n+j}\nu_j-\omega_j\nu_{n+j})-x_0\rho}\vert\xi_j\vert^2dS(\underline\xi)\right)d\rho\\
=\frac{(2n-1)!\sigma_{2n-1}}{n}\int_{-1}^1\frac{(1-t^2)^{\frac{2n-3}{2}}}{(x_0-irt)^{2n}}dt\\
+(2n-2)!\sigma_{2n-1}\left(\frac{\vert z_j\vert^2}{r^2}-\frac{1}{n}\right)\int_{-1}^1\frac{(2nt^2-1)(1-t^2)^{\frac{2n-3}{2}}}{(x_0-irt)^{2n}}dt.
\end{multline*}
Using equalities (i) and (iii) of Lemma \ref{tresintegrales}, we obtain
\begin{equation}\label{I1pparte}
\int_0^\infty\rho^{2n-1}\left(\int_{S^{2n-1}}e^{ir\rho\sum_{j=1}^n(\omega_{n+j}\nu_j-\omega_j\nu_{n+j})-x_0\rho}\,\vert\xi_j\vert^2dS(\underline\xi)\right)d\rho=\lambda_n\left(\frac{\vert z_j\vert^2}{r^2}G-H\right).
\end{equation}
Similarly, we have
\begin{equation}\label{I1sparte}
\int_0^\infty\rho^{2n-1}\left(\int_{S^{2n-1}}e^{ir\rho\sum_{j=1}^n(\omega_{n+j}\nu_j-\omega_j\nu_{n+j})-x_0\rho}\,\overline\xi_j\xi_k\,dS(\underline\xi)\right)d\rho=\lambda_n\frac{\overline z_jz_k}{r^2}G.
\end{equation}
From (\ref{I1pparte}) and (\ref{I1sparte}) it then follows that
\[I_3=\lambda_n\left(\frac{\underline z^\dagger\underline z}{r^2}G-\beta H\right).\]
Finally, on account of (\ref{dualexp}) we see that $I_4=I_0-I_3$.
\end{proof}
\noindent
As we have already pointed out, plane wave exponential function (\ref{realexpwh}) satisfies $\mathbb DP=0=P\mathbb D$ and hence so does the function $I$. In the case one might be interested in checking this directly it is then convenient to write $I$ as  
\begin{align*}
I&=(f_0f_0^{\dagger}I_3-f_0^{\dagger}I_1)f_0+(f_0^{\dagger}f_0I_4+f_0I_2)f_0^{\dagger}\\
&=f_0(f_0^{\dagger}f_0I_3-f_0^{\dagger}I_2)+f_0^{\dagger}(f_0f_0^{\dagger}I_4+f_0I_1).
\end{align*}
The explanation for doing this, is that the two functions between brackets in the first equality are solutions of $\mathbb Df=0$, while the other two in the second equality satisfy $f\mathbb D=0$. Moreover, if we take for example the first function $f_0f_0^{\dagger}I_3-f_0^{\dagger}I_1$, then the condition $\mathbb Df=0$ will imply, among other relations, that  
\begin{equation}\label{relGH}
\partial_{x_0}G=\frac{(2n+1)r^2}{(x_0^2+r^2)^\frac{2n+3}{2}},\quad\partial_{x_0}H=\frac{2}{(x_0^2+r^2)^\frac{2n+1}{2}}.
\end{equation}
These relations will be useful later.

\begin{rem}
At first glance, it seems that the function $I$ can be extended to the whole space $\mathbb R^{2n+1}$ except for the line $\{(x_0,\underline x):\;\underline x=0\}$, reflecting the nature of functions $G$ and $H$. But a closer look reveals that these singularities are in fact removable on the upper half-line $\{(x_0,\underline x):\;x_0>0,\;\underline x=0\}$. This remains valid for all $n\ge2$ and for the case $n=1$ we simply encounter a point singularity at the origin, because for that particular value we have  
\[\frac{\underline z^\dagger\underline z}{r^2}G-\beta H=\frac{f_1^{\dagger}f_1x_0}{(x_0^2+r^2)^{\frac{3}{2}}}.\]
\end{rem}

\begin{defn}
The function $\mathsf{E}$ defined by 
\begin{multline*}
\mathsf{E}(x_0,\underline x)=\frac{1}{\sigma_{2n+1}}\left[-\frac{\underline z^{\dagger}}{(x_0^2+\vert \underline x\vert^2)^{\frac{2n+1}{2}}}+f_0\left(\beta H-\frac{\underline z^\dagger\underline z}{\vert \underline x\vert^2}G\right)\right.\\
\left.+f_0^{\dagger}\left(-\frac{x_0}{(x_0^2+\vert \underline x\vert^2)^{\frac{2n+1}{2}}}-\beta H+\frac{\underline z^\dagger\underline z}{\vert \underline x\vert^2}G\right)+f_0^{\dagger}f_0\left(\frac{\underline z^{\dagger}-\underline z}{(x_0^2+\vert \underline x\vert^2)^{\frac{2n+1}{2}}}\right)\right],\quad\beta=\sum_{j=1}^nf_j^\dagger f_j,
\end{multline*}
which satisfies $\mathbb D\mathsf{E}=0=\mathsf{E}\mathbb D$ in $\mathbb R^{2n+1}\setminus\{(x_0,\underline x):\;x_0\le0,\;\underline x=0\}$, will be called the Cauchy kernel for the $h$-submonogenic system. 
\end{defn}

\section{An integral representation formula for $\mathbb Df=0$}\label{sect5}

This last section of the paper is devoted to justify why we named function $\mathsf E$ the Cauchy kernel for the $h$-submonogenic system. We begin by stating a divergence theorem for the operator $\mathbb D$.

Suppose that $\Omega$ is a simply connected bounded and open set in $\mathbb R^{2n+1}$ with a piecewise smooth boundary denoted by $\partial\Omega$. For each $(x_0,\underline x)=(x_0,\dots,x_{2n})\in\partial\Omega$, we denote by 
\[\big(\nu_0(x_0,\underline x),\dots,\nu_{2n}(x_0,\underline x)\big)\]  
the outward unit normal to $\partial\Omega$ at this point. Furthemore, we shall also write
\begin{align*}
\underline\nu(x_0,\underline x)&=\sum_{j=1}^n\big(\nu_j(x_0,\underline x)+i\nu_{n+j}(x_0,\underline x)\big)f_j,\\
\underline\nu^{\dagger}(x_0,\underline x)&=\big(\underline\nu(x_0,\underline x)\big)^{\dagger}=\sum_{j=1}^n\big(\nu_j(x_0,\underline x)-i\nu_{n+j}(x_0,\underline x)\big)f_j^{\dagger}.
\end{align*}
Using the classical divergence theorem
\[\int_{\partial\Omega}F\nu_jdS=\int_{\Omega}\partial_{x_j}FdV,\quad j=0,\dots,2n,\] 
we can easily deduce:

\begin{thm}\label{GOGthm}
If $F(x_0,\underline x)$ and $G(x_0,\underline x)$ are continuously differentiable functions in $\Omega\cup\partial\Omega$, then we have
\[\int_{\partial\Omega}F\,\mathsf{n}\,G\,dS=2\int_{\Omega}\big((F\mathbb D)G+F(\mathbb DG)\big)dV,\]
where $\mathsf{n}=f_0\nu_0+f_0^{\dagger}\nu_0+f_0f_0^{\dagger}\underline\nu+f_0^{\dagger}f_0\underline\nu^{\dagger}$. 
\end{thm}
\noindent
The main problem we face in willing to prove an integral representation formula is the fact that the Cauchy kernel $\mathsf{E}$ has a half-line of singularities. However, we overcome this difficulty by considering the function $\mathsf{K}=\partial_{x_0}\mathsf{E}$, which has a single singularity. Indeed, using (\ref{relGH}) we obtain
\begin{multline*}
\mathsf{K}(x_0,\underline x)=\frac{1}{\sigma_{2n+1}}\left[\frac{(2n+1)x_0\underline z^{\dagger}}{(x_0^2+\vert \underline x\vert^2)^{\frac{2n+3}{2}}}+f_0\left(\frac{2\beta}{(x_0^2+\vert \underline x\vert^2)^{\frac{2n+1}{2}}}-\frac{(2n+1)\underline z^\dagger\underline z}{(x_0^2+\vert \underline x\vert^2)^{\frac{2n+3}{2}}}\right)\right.\\
\left.+f_0^{\dagger}\left(\frac{2nx_0^2-\vert\underline x\vert^2}{(x_0^2+\vert \underline x\vert^2)^{\frac{2n+3}{2}}}-\frac{2\beta}{(x_0^2+\vert \underline x\vert^2)^{\frac{2n+1}{2}}}+\frac{(2n+1)\underline z^\dagger\underline z}{(x_0^2+\vert \underline x\vert^2)^{\frac{2n+3}{2}}}\right)+f_0^{\dagger}f_0\left(\frac{(2n+1)x_0(\underline z-\underline z^{\dagger})}{(x_0^2+\vert \underline x\vert^2)^{\frac{2n+3}{2}}}\right)\right].
\end{multline*}
If we deal with the above function instead of $\mathsf{E}$, we have the following result:

\begin{thm}\label{partIRF}
Suppose that $F(x_0,\underline x)$ is a continuously differentiable function in $\Omega\cup\partial\Omega$. If $F(x_0,\underline x)$ is $h$-submonogenic in $\Omega$, then
\[\partial_{x_0}F(u_0,\underline u)=\int_{\partial\Omega}\mathsf{K}(x_0-u_0,\underline x-\underline u)\mathsf{n}(x_0,\underline x)F(x_0,\underline x)dS(x_0,\underline x),\]
for every $(u_0,\underline u)\in\Omega$.
\end{thm}
\begin{proof}
Assume that $(u_0,\underline u)\in\Omega$. We denote by $B_{\epsilon}(u_0,\underline u)$ the open ball  with center $(u_0,\underline u)$ and radius $\epsilon>0$. By Theorem \ref{GOGthm}, we have that 
\[\int_{\partial(\Omega\setminus B_{\epsilon}(u_0,\underline u))}\mathsf{K}(x_0-u_0,\underline x-\underline u)\mathsf{n}(x_0,\underline x)F(x_0,\underline x)dS(x_0,\underline x)=0\]
and consequently
\[\int_{\partial\Omega}\mathsf{K}\,\mathsf{n}\,F\,dS=\int_{\partial B_{\epsilon}(u_0,\underline u)}\mathsf{K}\,\mathsf{n}\,F\,dS.\]
The result will be proved if we verify that the integral on the right-hand side has limit $\partial_{x_0}F(u_0,\underline u)$ as $\epsilon$ tends to zero. First, note that $\nu_j(x_0,\underline x)=(x_j-u_j)/\epsilon$ for $(x_0,\underline x)\in\partial B_{\epsilon}(u_0,\underline u)$, and hence
\[\mathsf{n}(x_0,\underline x)=\frac{1}{\epsilon}\big(f_0(x_0-u_0)+f_0^{\dagger}(x_0-u_0)+f_0f_0^{\dagger}(\underline z-\underline z_0)+f_0^{\dagger}f_0\big(\underline z^{\dagger}-\underline z_0^{\dagger}\big)\big),\]
where $\underline z_0=\sum_{j=1}^n(u_j+iu_{n+j})f_j$. A direct computation then yields
\begin{multline*}
\mathsf{K}(x_0-u_0,\underline x-\underline u)\mathsf{n}(x_0,\underline x)=\frac{1}{\sigma_{2n+1}\epsilon^{2n+2}}\Big(\big(2nf_0^{\dagger}f_0+2\beta(f_0f_0^{\dagger}-f_0^{\dagger}f_0)\big)(x_0-u_0)\\
-\big((2n+1)-2\beta\big)f_0\big(\underline z^{\dagger}-\underline z_0^{\dagger}\big)-\big(1+2\beta\big)f_0^{\dagger}(\underline z-\underline z_0)\Big),\quad (x_0,\underline x)\in\partial B_{\epsilon}(u_0,\underline u).
\end{multline*}
This now implies
\begin{align*}
\int_{\partial B_{\epsilon}(u_0,\underline u)}\mathsf{K}\,\mathsf{n}\,F\,dS=\frac{1}{\sigma_{2n+1}\epsilon^{2n+1}}\left(\big(2nf_0^{\dagger}f_0+2\beta(f_0f_0^{\dagger}-f_0^{\dagger}f_0)\big)\int_{B_{\epsilon}(u_0,\underline u)}\partial_{x_0}FdV\right.\\
\left.-2\big((2n+1)-2\beta\big)f_0\int_{B_{\epsilon}(u_0,\underline u)}\partial_{\underline z}FdV-2\big(1+2\beta\big)f_0^{\dagger}\int_{B_{\epsilon}(u_0,\underline u)}\partial_{\underline z^{\dagger}}FdV\right),
\end{align*}
where we have also used the classical divergence theorem. It is not difficult to check that the right-hand side of the last expression has limit
\begin{multline}\label{liguald}
\frac{1}{(2n+1)}\left(\big(2nf_0^{\dagger}f_0+2\beta(f_0f_0^{\dagger}-f_0^{\dagger}f_0)\big)\partial_{x_0}F(u_0,\underline u)\right.\\
\left.-2\big((2n+1)-2\beta\big)f_0\partial_{\underline z}F(u_0,\underline u)-2\big(1+2\beta\big)f_0^{\dagger}\partial_{\underline z^{\dagger}}F(u_0,\underline u)\right)
\end{multline}
as $\epsilon$ tends to zero. But since $F$ is $h$-submonogenic, it follows that $2f_0\partial_{\underline z}F=-f_0f_0^{\dagger}\partial_{x_0}F$ and $2f_0^{\dagger}\partial_{\underline z^{\dagger}}F=-f_0^{\dagger}f_0\partial_{x_0}F$. Using this, we may verify that (\ref{liguald}) is in fact equal to $\partial_{x_0}F(u_0,\underline u)$.    
\end{proof}

We end the paper by showing that for certain sets $\Omega$ and $h$-submonogenic functions $F(x_0,\underline x)$ it is possible to have a standard Cauchy integral representation formula.
 
\begin{thm}
Let $C_a(R)$ be the semi-infinite cylinder defined by  
\[C_a(R)=\left\{(x_0,\underline x)\in\mathbb R^{2n+1}:\;x_0<a,\;\vert\underline x\vert<R\right\}.\]
Suppose that $F(x_0,\underline x)$ is a continuously differentiable function in $C_a(R)\cup\partial C_a(R)$ such that $F$ is bounded and satisfies $\lim_{x_0\rightarrow-\infty}F(x_0,\underline x)=0$. If $F(x_0,\underline x)$ is $h$-submonogenic in $C_a(R)$, then
\[F(u_0,\underline u)=-\int_{\partial C_a(R)}\mathsf{E}(x_0-u_0,\underline x-\underline u)\mathsf{n}(x_0,\underline x)F(x_0,\underline x)dS(x_0,\underline x),\quad (u_0,\underline u)\in C_a(R).\]
\end{thm}
\begin{proof}
Let $(u_0,\underline u)\in C_a(R)$ and for $b<u_0$ we put 
\begin{align*}
C_{a,b}(R)&=\left\{(x_0,\underline x)\in\mathbb R^{2n+1}:\;b<x_0<a,\;\vert\underline x\vert<R\right\},\\
D_b(R)&=\left\{(x_0,\underline x)\in\mathbb R^{2n+1}:\;x_0=b,\;\vert\underline x\vert\le R\right\}.
\end{align*}
First, observe that the above integral is understood as 
\begin{multline*}
\int_{\partial C_a(R)}\mathsf{E}(x_0-u_0,\underline x-\underline u)\mathsf{n}(x_0,\underline x)F(x_0,\underline x)dS(x_0,\underline x)\\
=\lim_{b\rightarrow-\infty}\int_{\partial C_{a,b}(R)\setminus D_b(R)}\mathsf{E}(x_0-u_0,\underline x-\underline u)\mathsf{n}(x_0,\underline x)F(x_0,\underline x)dS(x_0,\underline x).
\end{multline*} 
From the equalities
\[\mathsf{E}(x_0,\underline x)=-\int_0^\infty\partial_{s}\mathsf{E}(x_0+s,\underline x)ds=-\int_0^\infty\mathsf{K}(x_0+s,\underline x)ds\]
we obtain 
\begin{multline*}
\int_{\partial C_{a,b}(R)\setminus D_b(R)}\mathsf{E}(x_0-u_0,\underline x-\underline u)\mathsf{n}(x_0,\underline x)F(x_0,\underline x)dS(x_0,\underline x)\\
=-\int_0^\infty\left(\int_{\partial C_{a,b}(R)\setminus D_b(R)}\mathsf{K}(x_0-u_0+s,\underline x-\underline u)\mathsf{n}(x_0,\underline x)F(x_0,\underline x)dS(x_0,\underline x)\right)ds.
\end{multline*}
Theorem \ref{partIRF} now implies that 
\begin{multline*}
\int_{\partial C_{a,b}(R)\setminus D_b(R)}\mathsf{K}(x_0-u_0+s,\underline x-\underline u)\mathsf{n}(x_0,\underline x)F(x_0,\underline x)dS(x_0,\underline x)\\
=\partial_{u_0}F(u_0-s,\underline u)-\int_{D_{b}(R)}\mathsf{K}(x_0-u_0+s,\underline x-\underline u)\mathsf{n}(x_0,\underline x)F(x_0,\underline x)dS(x_0,\underline x)
\end{multline*}
and for that reason, using Lebesgue's dominated convergence theorem,
\[\lim_{b\rightarrow-\infty}\int_{\partial C_{a,b}(R)\setminus D_b(R)}\mathsf{K}(x_0-u_0+s,\underline x-\underline u)\mathsf{n}(x_0,\underline x)F(x_0,\underline x)dS(x_0,\underline x)=\partial_{u_0}F(u_0-s,\underline u).\]
We thus get that
\begin{multline*}
\int_{\partial C_a(R)}\mathsf{E}(x_0-u_0,\underline x-\underline u)\mathsf{n}(x_0,\underline x)F(x_0,\underline x)dS(x_0,\underline x)=-\int_0^\infty\partial_{u_0}F(u_0-s,\underline u)ds\\
=\int_0^\infty\partial_{s}F(u_0-s,\underline u)ds=-F(u_0,\underline u),
\end{multline*}
which completes the proof. 
\end{proof}

\subsection*{Acknowledgments}

D. Pe\~na Pe\~na acknowledges the support of a Postdoctoral INdAM Fellowship cofunded by Marie Curie actions.

\end{document}